\documentclass[11pt]{amsart}
\usepackage{amsmath}
\usepackage{amsfonts}
\usepackage{amssymb}
\newtheorem{thm}{Theorem}

\newtheorem{cor}[thm]{Corollary}

\newtheorem{rem}[thm]{Remark}

\newtheorem{prop}[thm]{Proposition}
\newtheorem{prob}[thm]{Problem}

\newtheorem{defn}[thm]{Definition}
\newtheorem{conj}[thm]{Conjecture}

\newcommand{\ww}{\omega}
\newcommand{\A}{\cl A_{\omega}}
\newcommand{\BZ}{\beta \bb Z}
\newcommand{\BG}{\beta G}

\newcommand{\bb}[1]{\mathbb{#1}}
\newcommand{\cl}[1]{\mathcal{#1}}
\newcommand{\ff}[1]{\mathfrak{#1}}

\begin{document}

\title[Kadison-Singer]{A Dynamical Systems Approach to the Kadison-Singer Problem}

\author[V.~I.~Paulsen]{Vern I.~Paulsen}
\address{Department of Mathematics, University of Houston,
Houston, Texas 77204-3476, U.S.A.}
\email{vern@math.uh.edu}

\date{\today}

\thanks{This research was supported in part by NSF grant
  DMS-0600191. Portions of this research were conducted at the
  American Institute of Mathematics.}

\begin{abstract}  
In these notes we develop a link between the Kadison-Singer problem
and questions about certain dynamical systems. We conjecture that
whether or not a given state has a unique extension is related to
certain dynamical properties of the state. We prove that if any state
corresponding to a minimal idempotent point extends uniquely to the
von Neumann algebra of the group, then every state extends uniquely to
the von Neumann algebra of the group. We prove that if any state
arising in the Kadsion-Singer problem has a unique extension, then the
injective envelope of a C*-crossed product algebra associated with the
state necessarily contains
the full von Neumann algebra of the group. We prove that this latter
property holds for states arising from rare ultrafilters and
$\delta$-stable ultrafilters, independent, of the group action and
also for states corresponding to non-recurrent points in the corona of the group.
\end{abstract}

\maketitle


\section{Introduction}

Let $\cl H$ be a separable Hilbert space, and let $\cl D \subset B(\cl
H)$ be a discrete MASA.
The Kadison-Singer problem \cite{KS} asks whether or not  
every pure state on $\cl D$ has a unique extension to a state on
$B(\cl H)$.
Without loss of generality, one can assume that the Hilbert space is
$\ell^2(\bb N)$, where $\bb N$ denotes the natural numbers, with the canonical orthonormal basis, $\{ e_n \}_{n
  \in \bb N}$ and that the MASA is the subalgebra of operators that
are diagonal with respect to this basis.  

However, since any two
discrete MASA's on any two separable infinite dimensional Hilbert
spaces are conjugate, one may equally well assume that the Hilbert
space is $\ell^2(G)$, where $G$ is a countable, discrete group, with canonical orthonormal basis $\{ e_g
\}_{g \in G}$ and that the MASA is the set of operators that
are diagonal with respect to this basis.
Thus, $\cl D = \{ M_f: f \in \ell^{\infty}(G) \},$ where $M_f$ denotes the operator of multiplication by the function $f.$ We let $1$ denote the identity of $G$ and let $U_g= \lambda(g)$ denote the unitary operators given by the left regular representation, so that $U_g e_h = e_{gh}.$

The reason that we prefer this slight change of perspective, is that
we are interested in incorporating properties of the group action into
results on the Kadison-Singer problem.
Indeed, identifying the MASA, $\cl D \equiv \ell^{\infty}(G) \equiv
C(\BG)$, where $\BG$ denotes the Stone-Cech compactification of $G$, then pure states on $\cl D$ correspond to the homomorphisms
induced by evaluations at points in $\BG$.

Moreover, we have that for each $g \in G$, the map $g_1 \to gg_1$ extends uniquely to a
homeomorphism, $h_g: \BG \to \BG$ and this family of homeomorphisms satisfy, $h_{g_1} \circ h_{g_2}= h_{g_1g_2},$ that is, they induce an action of $G$ on the space $\BG$ and we set $h_g(\omega) = g \cdot \omega.$

In this paper we study the extent to which uniqueness or
non-uniqueness of extensions of the pure state induced by
a point $\ww \in \BG$ is related to the dynamical properties of the point.
In particular, we will be interested in the orbit $G \cdot \omega = \{
g \cdot \omega : g \in G \}$ of the point.

We will use the fact that the map $g \to g \cdot \ww$ is one-to-one,
which is a consequence of a theorem of Veech\cite{Ve}.
To see why this is so, note that this map being one-to-one is equivalent to requiring that the stabilizer subgroup, $G_{\ww}= \{g \in G: g \cdot \ww = \ww \}$ consist of the identity. When this is the case, we shall say that $\ww$ has {\bf trivial stabilizer.}

If every point in $\BG$ has trivial stabilizer, then $G$ acts
freely(and continuously) on $\BG$. Conversely, if there exists a
compact, Hausdorff space $X$ equipped with a continuous $G$-action
that is free, then choosing any $x \in X,$ the map $g \to g \cdot x$
extends to a continuous $G$-equivariant map $h:\BG \to X$ and for any
point $\ww \in \BG,$ the stabilizer of $\ww$ is contained in the
stabilzer of $h(\ww)$ and hence is trivial. Thus, every point in $\BG$
has trivial stabilizer if and only if $G$ can act freely and
continuously on some compact Hausdorff space. Veech's theorem\cite{Ve}
shows that, in fact, every locally compact group acts freely on a
compact, Hausdorff space. Thus, every point in $\BG$ has trivial stabilizer.

We begin by examining properties of any state extension of the pure
state given by evaluation at a point $\ww.$ To this end, given $M_f \in \cl D,$
we let $f \in C(\BG)$ denote the corresponding continuous function
on $\BG$ and let $s_{\ww}: \cl D \to \bb C$ denote the pure state
given by evaluation at $\ww$, that is, $s_{\ww}(f) = f(\ww).$

Let $s: B(\ell^2(G)) \to \bb C$ be any state
  extension of $s_{\ww}$ and let $\pi:B(\ell^2(G)) \to B(\cl H_s)$
  and $v_1 \in \cl H_s$ be the GNS representation of $s$, so that
  $s(X)= \langle \pi(X)v_1,v_1 \rangle.$ 
We set $v_g = \pi(U_g)v_1, g \in G$, let $\cl L_s \subseteq \cl
  H_s$ denote the closed linear span of the $v_g$'s and let
  $\phi_s:B(\ell^2(G)) \to B(\cl L_s)$ denote the completely
  positive map given by $\phi_s(X) = P_{\cl L_s} \pi(X) \mid_{\cl
  L_s}.$

Note that for any $M_f \in \cl D$ we have that $\langle \pi(M_f)v_g, v_g
\rangle= \langle \pi(U_g^{-1}M_fU_g)v_1,v_1 \rangle = f(g \cdot \ww).$ Hence,
these vectors are reducing for $\pi(\cl D)$ and orthonormal. Hence,
they are an orthonormal basis for $\cl L_s$ and the map $We_g= v_g$ is a Hilbert space isomorphism between $\cl L_s$ and $\ell^2(G).$  
Setting $\psi_s(X)= W^*\phi_s(X)W$ we obtain a completely positive map on $B(\ell^2(G)).$

Note that we have that $\langle \psi_s(X) e_h,e_g \rangle = \langle \pi(X) v_h, v_g \rangle = \langle \pi(U_{g^{-1}}XU_h) v_1,v_1 \rangle = s(U_{g^{-1}}XU_h).$
This shows that the correspondence $s \to \psi_s$ is one-to-one.

In particular, we have that $\psi_s(U_g)=U_g$ and
that for any $M_f \in \cl D, \psi_s(M_f) = \pi_{\ww}(M_f),$ where
$\pi_{\ww}: \cl D \to \cl D$ is the *-homomorphism given by
$$\pi_{\ww}(M_f)(g) = f(g \cdot \ww).$$

From these two equations, we see that the restriction of $\psi_s$ to
the C*-algebra generated by $\cl D$ and the set $\{ U_g: g \in G \}$,
is a *-homomorphism, which
we will denote by $\pi_{\ww} \times \lambda,$ satisfying,
$$\pi_{\ww}\times \lambda(\sum_n M_{f_n} U_{g_n}) = \psi_s( \sum_n M_{f_n} U_{g_n}) = \sum_n \pi_{\ww}(M_{f_n}) U_{g_n},$$
for every finite, or norm convergent sum.
This algebra is, in fact, *-isomorphic to the {\em reduced crossed-product
  C*-algebra}, $\cl D \times_r G$ as defined in say \cite{Pe},
although we do not need that fact here, we shall adopt that notation
for the algebra.

By Choi's theory of multiplicative domains\cite{Ch} of completely
positive maps, we have that for any $A_1,A_2 \in \cl D \times_r G$ and $X
\in B(\ell^2(G))$, we have
that $$\psi_s(A_1XA_2)= \pi_{\ww}\times \lambda(A_1) \psi_s(X)
\pi_{\ww}\times \lambda(A_2).$$

We now characterize the range of this correspondence.

\begin{thm} Let $\ww \in \BG.$ If $\psi: B(\ell^2(G)) \to
  B(\ell^2(G))$ is any completely positive extension of
  $\pi_{\ww}\times \lambda$
  and we set $s(X) = \langle \psi(X) e_1,e_1 \rangle ,$ then $s$ is a
  state extension of $s_{\ww}$ and $\psi = \psi_s .$ Consequently, the
  map, $s \to \psi_s$ is a one-to-one, onto affine map between the
  convex set of state extensions of $s_{\ww}$ and the convex set of
  completely positive extensions of $\pi_{\ww} \times \lambda.$
\end{thm}

\begin{proof} It is clear that $s$ is a state extension of $s_{\ww}.$ Now given any $X \in B(\ell^2(G))$ and $g,h \in G,$ we have that $\langle \psi(X) e_h, e_g \rangle = \langle \psi(U_{g^{-1}}XU_h) e_1,e_1 \rangle = s(U_{g^{-1}}XU_h) = \langle \psi_s(X) e_h, e_g \rangle $ and, thus, $\psi(X) = \psi_s(X).$

Hence, the map $\psi \to s$ defines an inverse to the map $s \to
\psi_s$, so that these correspondences are one-to-one and
onto. Finally, it is clear that both of these correspondences preserve
convex combinations.
\end{proof}

\begin{cor} Let $\ww \in \BG.$ Then the following are equivalent:
\begin{itemize}
\item $s_{\ww}: \cl D \to \bb C$ has a unique extension to a
  state on $B(\ell^2(G)),$
\item $\pi_{\ww} \times \lambda: \cl D \times_r G \to B(\ell^2(G))$ has a unique
  extension to a completely positive map on $B(\ell^2(G)).$
\end{itemize}
\end{cor}

There is, of course, always one distinguished state extension of
$s_{\ww}$. If we let $E_0: B(\ell^2(\bb Z)) \to \cl D$ be the
canonical projection onto the diagonal, then the {\bf regular
  extension} of $s_{\ww}$ is given by
$$X \to s_{\ww}(E_0(X)).$$
Corresponding to this regular extension is a unique completely
positive map, $\psi_{\ww}:B(\ell^2(\bb Z)) \to B(\ell^2(\bb Z))$ which we
also call the {\bf regular completely positive extension.}
We wish to describe this map in some detail.

Every $X \in B(\ell^2(G))$ has a formal series, $X \sim \sum_{g \in
  G} M_{f_g} U_g$ where $M_{f_g}= E_0(X U_g^{-1}).$ To compute the $(g_i,g_j)$-th
  entry of $\psi_{\ww}(X),$ we note that
$$\psi_{\ww}(X)_{g_i,g_j}= \langle \psi_{\ww}(X)e_{g_j},e_{g_i} \rangle = \langle
\psi_{\ww}(U_{g_i}^{-1}X U_{g_j})e_1,e_1 \rangle = s_{\ww}(E_0(U_{g_i}^{-1}X U_{g_j})).$$
But, we have that $U_{g_i}^{-1}X U_{g_j} \sim \sum_{g \in G}
U_{g_i}^{-1}M_{f_g} U_{g_i}(U_{g_i^{-1}g g_j}),$
and hence, $E_0(U_{g_i}^{-1}X U_{g_j}) = U_{g_i}^{-1} M_{f_{g_ig_j^{-1}}} U_{g_i}.$
Thus, $\psi_{\ww}(X)_{g_i,g_jj} = f_{g_i g_j^{-1}}(g_i \cdot \ww).$

These observations lead to the following result.

\begin{thm} Let $\ww \in \BG$ and let $\psi_{\ww}$ be the regular
  completely positive extension corresponding to $s_{\ww}$. If $X \sim
  \sum_{g \in G} M_{f_g} U_g,$ then $\psi_{\ww}(X) \sim \sum_{g \in G}
  \pi_{\ww}(M_{f_g}) U_g$. Moreover, $s_{\ww}$ has a unique extension to a state
  on $B(\ell^2(G))$ if and only if $\psi_{\ww}$ is the unique
  completely positive map on $B(\ell^2(G))$ extending $\pi_{\ww}
  \times \lambda:
  \cl D \times_r G \to B(\ell^2(G)).$
\end{thm}

It may seem paradoxical to attempt to make progress on the
Kadison-Singer problem by replacing statements about uniqueness 
of the extension of a state to statements about uniqueness of the
extension of a completely positive map, but something is gained by
making the domain and range of the map the same space. We make this precise in the following results.

\begin{defn} Let $\ww \in \BG$, then we define the uniqueness set for $\ww$ to be the set,
$$\cl U(\ww) = \{ X \in B(\ell^2(G)): s(X)= s_{\ww}(E_0(X))  \forall s\},$$
where $s$ denotes an arbitrary state extension of $s_{\ww}.$
We also define the uniqueness set for $\pi_{\ww}$ to be the set,
$$\cl U(\pi_{\ww} \times \lambda) = \{ X \in B(\ell^2(G)): \psi(X) =
\psi_{\ww}(X) \forall \psi \},$$ where $\psi$ denotes an arbitrary
completely positive extension of $\pi_{\ww} \times \lambda.$
\end{defn}

\begin{prop}  Let $\ww \in \BG,$ then $$\cl U(\pi_{\ww} \times \lambda) = \bigcap_{g,h \in G} U_g \cl U(\ww) U_h.$$
\end{prop}
\begin{proof} Since every completely positive extension of $\pi_{\ww}
  \times \lambda$ is of the form $\psi_s$ for some state extension of $s_{\ww},$ we have that $X \in \ff U(\pi_{\ww}),$ if and only if $\psi_s(X) = \psi_{\ww}(X)$ for every extension $s$. But this is equivalent to $s(U_{g^{-1}}XU_h) = \langle \psi_s(X) e_h, e_g \rangle = \langle \psi_{\ww}(X) e_h, e_g \rangle= s(E_0(U_{g^{-1}}XU_h))$ for every $g,h \in G,$ which is equivalent to $U_{g^{-1}}XU_h \in \cl U(\ww)$ for every $g,h \in G,$ and the result follows.
\end{proof} 

Of course, if $s_{\ww}$ has a unique extension, then the sets above
are all equal to $B(\ell^2(G)),$ but if some $s_{\ww}$ fails to have a
unique extension, then it should be easier to show that $\cl
U(\pi_{\ww} \times \lambda) \ne B(\ell^2(G)),$ than to show that $\cl U(\ww) \ne B(\ell^2(G)).$

We now turn our attention to some results that relate uniqueness of extension to injective envelopes.

\begin{defn} We let $\cl A_{\ww}=
\pi_{\ww} \times \lambda(\cl D \times_r G)$ and we denote the von Neumann algebra generated by $\{U_g: g \in G \}$ by $VN(G).$
\end{defn}

Note that $VN(G)$ is always contained in the range of $\psi_{\ww}.$

Recall that a map $\phi$ is called a {\bf $\cl B$-bimodule map} for an algebra $\cl B$ if $\phi(b_1xb_2)= b_1 \phi(x)b_2,$ for every $b_1,b_2 \in \cl B.$

\begin{prop}\label{condition} Let $\ww \in \BG.$  If $s_{\ww}$ has a unique extension,
  then every completely positive map, $\phi:B(\ell^2(G)) \to
  B(\ell^2(G))$ that fixes $\cl A_{\ww}$ elementwise, also fixes the
  range of $\psi_{\ww}$ elementwise and is a $VN(G)$-bimodule map.
\end{prop}
\begin{proof} If $\phi$ does not fix the range, then $\phi \circ \psi_{\ww}$ would be another
  completely positive map extending $\pi_{\ww}.$ Thus, $\phi$ must fix the range of $\psi_{\ww}$ elementwise. But $VN(G)$ is a subset of the range of $\psi_{\ww}$ and so must be fixed. By Choi's\cite{Ch} theory of multiplicative domains this implies that $\phi$ is a $VN(G)$-bimodule map.
\end{proof}

\begin{rem} For many countable groups $G$, even when $G= \bb Z,$ there exist completely positive maps, $\phi: B(\ell^2(G))
  \to B(\ell^2(G))$ that fix  $C^*(G)$ elementwise and hence are $C^*(G)$-bimodule
  maps, but whose range does not contain $VN(G)$ and that are not
  $VN(G)$-bimodule maps. Examples of such completely positive maps are
  constructed in \cite{CEKPT}. However, any completely positive map $\phi:
  B(\ell^2(G)) \to B(\ell^2(G))$ that
  fixes $\cl D \times_r G$ elementwise is necessarily the identity map on all of
  $B(\ell^2(G)),$ since $\cl D \times_r G$ contains the compact
  operators. Thus, since $C^*(G) \subseteq \cl A_{\ww} \subseteq \cl D
  \times_r G,$ whether or not $s_{\ww}$ has a unique extension should be
  related to how small $\pi_{\ww}(\cl D)$ can be made.
\end{rem}

This last result can be interpreted in terms of injective
envelopes. Recall, that given a unital C*-subalgebra $\cl A \subseteq B(\cl H)$
and a completely positive idempotent map $\phi:B(\cl H) \to
B(\cl H)$ that fixes $\cl A$ elementwise and is minimal among all such
maps, then the range of $\phi, \cl R(\phi)$ is completely
isometrically isomorphic to the injective envelope of $\cl A, I(\cl
A).$ Such maps are called {\it minimal $\cl A$-projections.}
Thus, in particular, the ranges of any two minimal $\cl A$-projections are completely isometrically isomorphic via a map that fixes $\cl
A$ elementwise. See \cite{Pa} for further details.
For this reason the collection of subspaces that are ranges of minimal
$\cl A$-projections are called the {\it copies of the injective
  envelope of $\cl A$.} If we let $\cl F(\cl A)$ denote the set of elements of $B(\cl H)$ that
are elementwise fixed by every completely positive map that fixes
$\cl A,$ then it is clear that $\cl F(\cl A)$ is contained in every
copy of $I(\cl A).$ In \cite{Pa2}  it is shown that $\cl F(\cl A)$ is, in fact,
the intersection of all copies of $I(\cl A),$ but that is not
necessary for the following result.

\begin{cor} \label{condition2} Let $\ww \in \BG.$  If $s_{\ww}$ has a unique extension,
  then $VN(G) \subseteq \cl R(\psi_{\ww}) \subseteq \cl F(\A).$
\end{cor}

Thus, if the Kadison-Singer problem has an affirmative answer, then
necessarily $VN(G) \subseteq I(\cl A_{\ww}),$ and this inclusion is as
a subalgebra for
every $\ww.$ However, we know that, generally, $VN(G)$ is {\it not} a
subalgebra of $I(C^*(G)).$ 

\begin{prob} Let $D_g \in \pi_{\ww}(\cl D)$ be chosen such that $Y \sim \sum_{g \in G} D_g U_g$ is a bounded operator. Is $Y$ necessarily in the range of $\psi_{\ww}$? Can conditions on the orbit of $\ww$ be given that guarantee that this is the case?
\end{prob}


\section{Dynamics and Algebra}

In this section we begin to look at how dynamical properties of points
and their behavior with respect to a natural semigroup structure on
$\BG$ can be related to uniqueness of extension. For more on this
structure see \cite{HS}, but we recall a few basic definitions.

Given $\omega \in \BG$ we let $\rho_{\omega}: \BG \to \BG$ be the
unique continuous function satisfying, $\rho_{\omega}(g) =
 g \cdot \omega,$ for all $g \in G$ so that $\rho_{\omega}(\BG)$
is the closure of the orbit of $\omega.$ Since $\rho_{\omega} \circ
h_g (g_1) = \rho_{\omega}(gg_1) = gg_1 \cdot \omega = h_g \circ \rho_{\omega} (g_1),$ we
have that $\rho_{\omega} \circ h_g = h_g \circ \rho_{\omega},$ i.e.,
the map $\rho_{\omega}$ is equivariant for the action of $G$ on $\BG.$ This map also defines a
semigroup structure on $\BG$ by setting, $\alpha \cdot \omega \equiv
\rho_{\omega}(\alpha).$ We caution that in spite of the notation,
this operation is {\em not} abelian even when the underlying group is abelian(except for finite groups). 

However, it is associative and
continuous in the left variable and so it gives $\BG$ the structure of
a {\em compact right topological semigroup.}
We refer the reader to \cite{HS} for these and other basic facts about
this algebraic structure on $\BG$. One fact that we shall use is that
the corona, $G^*= \BG \setminus G$ is a two-sided ideal in $\BG.$

We now wish to relate dynamical properties of a point $\ww$, to the
structure of $\A$ and to the semigroup properties of $\ww.$

\begin{prop}
Let $\omega, \alpha \in \BG.$ Then $\rho_{\omega} \circ
\rho_{\alpha} = \rho_{\alpha \cdot \omega},$ $\pi_{\alpha} \circ
\pi_{\omega} = \pi_{\alpha \cdot \omega}$ and $\psi_{\alpha} \circ \psi_{\omega} = \psi_{\alpha \cdot \omega}.$
\end{prop}
\begin{proof} We have that $\rho_{\omega} \circ \rho_{\alpha}(g)
  =\rho_{\omega}( h_g(\alpha))= h_g \circ
  \rho_{\omega}(\alpha) = h_g(\alpha \cdot \omega) = \rho_{\alpha
  \cdot \omega}(g),$ and hence, $\rho_{\omega} \circ \rho_{\alpha} =
  \rho_{\alpha \cdot \omega}.$

The second equality comes from the fact that after identifying $\cl D
= C(\BG),$ then $\pi_{\omega}(f) = f \circ \rho_{\omega},$ and
hence, $\pi_{\alpha} \circ \pi_{\omega}(f) = f \circ \rho_{\omega}
\circ \rho_{\alpha} = f \circ \rho_{\alpha \cdot \omega} =
\pi_{\alpha \cdot \omega}(f).$ The proof of the third identity is similar.
\end{proof}

\begin{defn}
A point $\omega \in \BG$ is called {\bf idempotent} if $\omega \cdot
\omega = \omega.$
\end{defn}

Non-zero idempotent points are known to exist \cite{HS}, since the
corona is a compact right continuous semigroup. As we shall see
below they play a special role in the Kadison-Singer problem.

\begin{prop} Let $\omega \in \BG.$ Then the following are equivalent:
\begin{itemize}
\item  $\omega$ is idempotent,
\item  $\rho_{\omega} \circ \rho_{\omega} = \rho_{\omega},$
\item  $\pi_{\omega}: \cl D \to \cl D$ is idempotent,
\item  $\pi_{\omega} \times \lambda: \cl D \times_r G \to \cl D \times_r G$ is idempotent,
\item $\psi_{\ww}$ is an idempotent completely positive map.
\end{itemize}
\end{prop}

\begin{thm} Let $\omega \in \BG$ be an idempotent. If $s_{\ww}$ has a unique state extension, then $\cl
  R(\psi_{\ww})$ is completely isometrically isomorphic to the injective envelope of $\cl
  A_{\omega}$. Moreover, $\cl R(\psi_{\ww})$ is the unique copy of
  the injective envelope inside $B(\ell^2(G))$ and $\psi_{\ww}$ is the
  unique projection onto it. Thus, the identity map on $\cl
  A_{\omega}$ extends uniquely to an embedding of $I(\cl A_{\omega})$
  into $B(\ell^2(G)).$
\end{thm}
\begin{proof} We have seen earlier that for any $\omega \in \BG$ the range of $\psi_{\ww}$ is contained in any copy of the injective envelope. Thus, when $\ww$ is idempotent, since $\psi_{\ww}$ is already a completely positive projection, it's range must be a minimal completely positive projection and it's range must be the unique copy of the injective envelope.
\end{proof}

\begin{conj} We conjecture that if $\ww$ is idempotent, then $s_{\ww}$
  posesses non-unique extensions. That is, the Kadison-Singer
  conjecture is false and idempotent points provide
  counterexamples. In fact, we believe that idempotent points fail to
  satisfy the condition of Corollary~\ref{condition2}, $VN(G) \subseteq \cl F(\cl A_{\ww}).$
\end{conj}

The following result lends some credence to the above conjecture, at
least for {\bf minimal idempotents.} An idempotent $\ww$ is minimal if it is
minimal in any of several different orders \cite[Definition~1.37]{HS}
and this is shown to be equivalent to the left ideal generated by
$\ww,$ i.e., $\{\alpha \cdot \ww: \alpha \in \beta G\}$ being a minimal left ideal
\cite[Theorem~2.9]{HS}. Moreover, by \cite[Theorem~19.23c]{HS}, a
minimal idempotent is {\bf uniformly recurrent}(see also \cite{Bl}
where this is proved for the case of $\bb N$).  Recall that $\ww$ uniformly
recurrent means \cite[Definition~19.1]{HS} that for every neighborhood $U$ of $\ww$ we have that
the set $S= \{g \in G: g \cdot \ww \in U \}$ is {\bf syndetic}, i.e.,
there exists a finite set $g_1,...,g_m \in G$ such that $g_1\cdot S
\cup g_2 \cdot S \cup \cdots \cup g_m \cdot S =G.$

\begin{thm} Let $G$ be a countable, abelian, discrete group and let $X \in
  VN(G).$ If there exists a minimal idempotent $\ww$ such that $X \in
  \cl U(\ww),$ then $X \in \cl U(\alpha)$ for every $\alpha \in \beta
  G.$
\end{thm}

Thus, if there exists {\it any} state $s_{\alpha}$ which fails to have
unique extension for some $X \in VN(G),$ then every minimal
idempotent fails to have unique extension for that $X$.

\begin{proof} Since $\cl U(\ww)$ is an operator system, $X \in \cl
  U(\ww)$ if and only if $Re(X)= (X+X^*)/2$ and $Im(X) = (X-X^*)/(2i)$
  are both in $\cl U(\ww).$ Thus, it will be sufficient to assume that
  $X=X^*.$ Moreover, since $I \in \cl U(\ww)$, it is sufficient to
  assume also that $E(X)=0.$  Now by Anderson's paving
  results\cite{An1}(see also \cite[Theorem~2.7]{PR}), $X \in
  \cl U(\ww)$ if and only if for each $\epsilon >0,$ there exists $A$
  in the ultrafilter corresponding to $\ww$ such that $-\epsilon P_A
  \le P_AXP_A \le + \epsilon P_A$ where $P_A \in \cl D$ is the
  diagonal projection $P_A = diag(a_g)$ with 
$$a_g = \begin{cases} 1 & g \in A\\0 & g \notin A \end{cases}.$$

Let $U \subseteq \beta G$ be the clopen neighborhood of $\ww$
satisfying $A= U \cap G,$ so that under the identification of $\cl D$
with $C(\beta G)$ the projection $P_A$ is identified with the
characteristic function of $U, \chi_U.$
We have that $S= \{ g \in G: g \cdot \ww \in U\}$ is syndetic, so let
$g_1,...,g_m$ be as in the definition of syndetic.

Note that $\psi_{\ww}(P_A) = \pi_{\ww}(P_A)= \pi_{\ww}(\chi_U) = diag(b_g)$ where 
$$b_g= \begin{cases} 1& g \cdot \ww \in U\\0 & g \cdot \ww \notin U \end{cases},$$
that is $\psi_{\ww}(P_A)= P_S.$ Also, since $X \in VN(G)$ we have that
$\psi_{\ww}(P_AXP_A) = \pi_{\ww}(P_A) \psi_{\ww}(X) \pi_{\ww}(P_A) =
P_S X P_S.$ Thus, applying $\psi_{\ww}$ to the above inequality, we have that $- \epsilon P_S \le P_SXP_S \le
+\epsilon P_S.$

Next notice that conjugating the first inequality by $\lambda(g),$
we have that $\lambda(g)P_A \lambda(g^{-1}) = P_{gA}$ and
$$\lambda(g)P_AXP_A\lambda(g^{-1}) = \lambda(g)P_A\lambda(g^{-1})
\lambda(g)X \lambda(g^{-1}) \lambda(g)P_A \lambda(g^{-1}) = P_{gA}XP_{gA}.$$
Thus, $-\epsilon P_{gA} \le P_{gA}XP_{gA} \le +\epsilon P_{gA}.$
Applying the map $\psi_{\ww}$ to this inequality and observing that
$\psi_{\ww}(P_{gA})=P_{gS},$ we have that
$$-\epsilon P_{g_iS} \le P_{g_iS}XP_{g_iS} \le +\epsilon P_{g_iS},$$
for $i=1,...,m.$

Since, $G= g_1S \cup \cdots \cup g_mS,$ and $\epsilon >0$ was
arbitrary, this shows that $X$ is pavable in Anderson's sense and so, 
applying Anderson's\cite{An3} paving results, we have that every pure
state on $\cl D$ has a unique extension to $X$.
\end{proof}

\begin{conj} The same result holds for non-abelian groups.
\end{conj}
\begin{conj} Assuming that
a minimal idempotent has a unique extension, should imply that every
  point has a unique extension. That is, if we fix any minimal
  idempotent $\ww$, then the Kadison-Singer conjecture
  is true if and only if $s_{\ww}$ has a unique state extension.  
\end{conj} 

If $\omega$ is any point in the corona, then $\psi_{\ww}$ annihilates any
operator all of whose ``diagonals'' are $c_0$. Thus, when $\omega$ is
idempotent, $\cl R(\psi_{\ww})$ can contain no operators with any $c_0$
``diagonals'' and so in particular, no compact operator. In this
sense, $\pi_{\ww}(\cl D)$ is a ``small'' subset of $\ell^{\infty}(G),$
which is one of the reasons that we conjecture that these points are
potential counterexamples. Thus, in this
case $\cl R(\psi_{\ww}) \cap \cl K = (0),$ where $\cl K$ denotes the
compacts.
Since any copy of the injective envelope of $\cl A_{\omega}$ is
contained in $\cl R(\psi_{\ww})$ it also follows that for
$\omega$ idempotent, $I(\cl A_{\ww}) \cap \cl K = (0)$ and hence cannot be all of
$B(\ell^2(G)).$

We now consider the opposite case, points for which $\pi_{\ww}(\cl D)$
is a ``large'' subset of $\ell^{\infty}(\cl D).$ We believe that these
points are good candidates for having unique extensions.

The following result characterizes the points for which $\cl K \cap
\cl R(\psi_{\ww}) \ne (0),$ at least for many groups and, in particular, for these points, we shall see that
the injective envelope of $\cl A_{\omega}$ is all of $B(\ell^2(G)).$

Recall that given a dynamical system, a point $\ww$ is {\em non-recurrent} if there is an open
neighborhood $U$ of the point such that $g \cdot \omega \notin U$ for all $g
\ne 1.$ 
Also, given a semigroup, an element $\omega$ is {\em right
  cancellative}, if $\alpha \cdot \omega = \beta \cdot \omega$ implies that
$\alpha = \beta.$

\begin{thm} \label{non-recurrent}Let $\omega \in \BG$ and assume that $G$ contains no non-trivial finite subgroups. Then the following are equivalent:
\begin{itemize}
\item[(i)]
$\omega$ is non-recurrent, 
\item[(ii)] $\cl K \cap \cl A_{\ww} \ne (0),$
\item[(iii)]  $ \cl K\subseteq \cl A_{\omega},$ 
\item[(iv)] $\pi_{\omega}: \cl D \to \cl D$ is onto,
\item[(v)] $\cl A_{\omega} = \cl D \times_r G,$
\item[(vi)] $\rho_{\omega}$ is one-to-one,
\item[(vii)] $\omega$ is right cancellative.
\end{itemize}
\end{thm}

\begin{proof} {\em (i) implies (iii).} If $\ww$ is non-recurrent there is a clopen neighborhood $U$ of $\ww$ containing no other $g \cdot \ww, g \ne 1$. 
Hence, $\pi_{\ww}(\chi_U)$ is the rank one projection onto the span of $e_1 .$ 
 
Since $U_g \in \A$, we have the matrix units, $E_{g_i,g_j} = U_{g_i} \pi_{\ww}(\chi_U) U_{g_j}^{-1} = \pi_{\ww}(U_{g_i} \chi_U U_{g_j}^{-1}) \in \A$ for all $g_i, g_j \in G$ and hence
$\cl K\subseteq \cl A_{\ww}$

\medskip

{\em (ii) implies (i).} Suppose $\A$ contains a nonzero compact $K=\sum \pi_{\ww} (M_{f_g})U_g$ (formal sum). Then all $\pi_{\ww} (M_{f_g})$ are compact. So there is a non-zero $\pi_{\ww} (M_f)\in \cl K\cap \A$. 
Applying a sequence $p_n$ of polynomials to this, tending to the characteristic functions of some eigenspace you find a finite rank "diagonal" projection $Q=\chi_S$  in $\A$ with $S$ a finite set. Choose such a $Q$ with $S= \{ g_1,...,g_n \}$ of minimal non-zero cardinality. Conjugating $Q$ by $U_g$, we obtain another such projection in $\A$ corresponding to the set $g \cdot S$ and so we may assume that $1 \in S.$ Also, the product of two such projections is also in $\A$ and is the finite rank projection corresponding to $S \cap gS.$ But since $S$ is of minimal non-zero cardinality, either $S \cap g \cdot S$ is empty or $S \cap g \cdot S = S.$ Hence, for each $g \in S, g^{-1} \cdot S =S,$ and so $S$ is a finite subgroup. Thus, $S= \{ 1 \},$ and we have that $E_{1,1} \in \A.$ 

Hence, there exists $f \in C(\BG)$, such that
$E_{1,1}=\pi_{\ww} (M_f)$. Now $f(\ww)=1$ and $f(g \cdot \ww)=0$ for all $g\ne 1$.  Thus, $U= \{ \alpha: |f(\alpha)| > 1/2 \}$ is an open set containing $\ww$ but no other $g \cdot \ww$ and  so $\ww$ is non-recurrent.

\medskip

Clearly, (iii) implies (ii) and so (i), (ii), and (iii) are equivalent.

Also, it is clear that (v) implies (iv) implies (ii). Moreover, since $\pi_{\ww}$ is given by composition with $\rho_{\ww},$ we have that (iv) and (vi) are equivalent. Since, $\alpha + \ww = \rho_{\ww}(\alpha),$ it is also clear that $\ww$ is right cancellative is equivalent to $\rho_{\ww}$ being one-to-one.

\medskip

{\em (i) implies (iv).}  Since $\ww$ is non-recurrent there is a neighborhood $U_1$ of $\ww$ that contains no other point on its orbit. Hence, $g \cdot U_1$ is a neighborhood of $g \cdot \ww$ containing no other point on the orbit. Thus, non-recurrent is equivalent to the set $\{ g \cdot \ww \}$ being discrete. Now using the fact that $\BG$ is a compact, Hausdorrf space and hence normal,  and that the set of points is countable, one can choose neighborhoods, $V_g$ of $g \cdot \ww$ such that $V_g \cap V_h$ is empty for $g \ne h,$ i.e., the points are what is called {\em strongly discrete.} To recall the construction, first enumerate the points, $\ww_i = g_i \cdot \ww,$ then using normality, choose for each $i$ disjoint open sets $U_i, V_i$ such that $\ww_i \in U_i$ and the closure of $\{ \ww_j: j \ne i \}$ is contained in $V_i.$  Then set $W_1=U_1, W_i = U_i \cap V_1 \cdots \cap V_{i-1}, i \ge 2.$

Thus, using the fact that $\BG$ is extremally disconnected, we may choose disjoint clopen sets $U_g, g \in G,$ with $g \cdot \ww \in U_g.$ We then have that $\pi_{\ww}(\chi_{U_g})$ is the rank one projection onto the span of $e_g.$
Given any $f_1 \in \ell^{\infty}(G)$ define $f_2 \in \ell^{\infty}(G),$ by $f_2(h) = f_1(g)$ if and only if $h \in U_g,$ and when $h \notin \cup_{g \in G} U_g$ set $f_2(h) = \lambda$ where $\lambda \in \bb C$ is any number not in the closure of $\{ f_1(g): g \in G \}.$ Then $\pi_{\ww}(f_2) = f_1$ and so (iv) holds.

\medskip

{\em (iv) implies (v).}  Given any finite sum, $B= \sum_{j=1}^n M_{f_j} U_{g_j},$ we may pick continuous functions $h_j,$ such that $\pi_{\ww}( \sum_{j=1}^n M_{h_j} U_{g_j} = B.$ Thus, the range of the *-homomorphism, $\pi_{\ww}: \cl D \times_r G \to \cl D \times_r G$ is dense in $\cl D \times_r G$ and hence is onto.

\end{proof}

\begin{cor} Let $G$ be a countable group with no non-trivial finite subgroups. If $\ww_1, \ww_2 \in \BG$ are non-recurrent, then $\ww_1 \cdot \ww_2$ is non-recurrent.
\end{cor}

\begin{cor} Let $G$ be a countable group with no non-trivial finite subgroups and let $\ww \in \BG.$ Then $I(\A)= B(\ell^2(G))$ if and only if $\ww$ is non-recurrent.
\end{cor}
\begin{proof} If $\ww$ is non-recurrent, then $\A$ contains all the compacts and hence any completely positive map that fixes $\A$ fixes all the compacts and hence is the identity map. To see this note that if it fixes every diagonal matrix unit, then it is a Schur product map, which fixes every matrix unit and so is the identity map.

Conversely, if $\ww$ is not non-recurrent, then $\A$ contains no
non-zero compacts and hence the quotient map to the Calkin algebra is
an isometry on $\A$. But the injective envelope is an essential extension of $\A$ and, hence, any embedding of $I(\A)$ into $B(\ell^2(G))$, composed with the quotient map must also be an isometry on  
$I(\A)$and hence, $I(\A) \ne B(\ell^2(G)).$ 
\end{proof}

The above result shows that non-recurrent points all satisfy the
condition of Corollary~\ref{condition2} that is necessary for
$s_{\ww}$ to have a unique extension. This lends credence to the
following conjectures.

\begin{conj} We conjecture that if $\ww \in \BG$ is non-recurrent, then
  the state $s_{\ww}$ extends uniquely to $B(\ell^2(G)).$
\end{conj}

\begin{conj} Let $G$ be a countable group with no finite subgroups. If
  $\ww \in \BG$ is non-recurrent, then we conjecture that $\cl R(\psi_{\ww}) =
  B(\ell^2(G)).$
\end{conj}

Assuming the last conjecture, one can prove that if $\ww$ is non-recurrent and $s_{\alpha}$ has a unique extension, then $s_{\alpha \cdot \ww}$ has a unique extension. So these conjectures might shed some light on the algebraic properties of points with unique extensions.


\section{Dynamical Properties of Ultrafilters}

We now examine the dynamical properties of various classes of ultrafilters that have been studied in relation to the Kadison-Singer problem. 

An ultrafilter $\beta$ on a countable set $N$ is called {\bf
  selective,}\cite{HS} if given any partition of $N = \cup P_i$ into subsets,
  either $P_i$ is in the ultrafilter for some $i$ or there exists a set $B$ in the ultrafilter, such that for every $i, B \cap P_i$ has cardinality at most $1.$  An ultrafilter is called {\bf rare}\cite{C} if for each partition $N= \cup P_i$ into finite subsets, there exsits a set $B$ in the ultrafilter such that for every $i, B \cap P_i$ has cardinality at most $1$. 
An ultrafilter is called a {\bf $\delta$-stable,}\cite{C} if for each
  partition $N= \cup P_i$, into sets of arbitrary sizes, then either
  one of the $P_i$'s is in the ultrafilter or there exists a set $B$
  in the ultrafilter such that for every $i, B \cap P_i$ is
  finite. Note that an ultrafilter is selective if and only if it is
  rare and $\delta$-stable.

Finally, a point in a topological space is called a {\bf
  P-point}\cite{Wa} if every $G_{\delta}$ that contains the point
  contains an  open neighborhood of the point. 

Note that if $X$ is a compact, Hausdorff space, $x \in X$ is a P-point
and $f \in C(X)$,  then $\{ y \in X: f(y) = f(x) \}$ is a $G_{\delta}$
set and hence contains an open neighborhood of $x$! Thus, this definition of P-point is antithetical to the concept of p-point that appears in function theory.  
Choquet\cite[Proposition 1]{C} proves that for any discrete set $N,\ww$ is a
$\delta$-stable if and only if $\ww$ is a P-point in the corona
$\beta N \setminus N.$ For this reason the term $\delta$-stable
has fallen into disuse and such ultrafilters are, generally, called P-points
without reference to the topological space.

Finally, we recall a stronger notion than non-recurrent. A point $\ww$ in a dynamical system is {\bf wandering}, if there exists a neighborhood, $U$ of $\ww$ such that $g \cdot U \cap h \cdot U$ is empty for any $g \ne h$ in the group.

In an earlier version of this paper, we proved that rare ultrafilters are wandering points in $\bb N$ and $\bb Z$ and conjectured that they were wandering in every group. Since then, Ken Davidson has verified this conjecture. We present our proof of the case of $\bb N$ and Davidson's proof for general groups below.

\begin{prop} Let $\ww \in \beta \bb N$ be rare, then $\ww$ is wandering in the corona $\bb N^*.$
\end{prop}

\begin{proof} Let $\ww$ be a rare ultrafilter on $\bb N.$ Consider the following partition of $ \bb N$
into finite sets:

\begin{multline*} \bb N=  \{  1\}\cup\{  2\}\cup\{ 3, 
4\}\cup\{ 5, 6\}\cup\{ 7, 8, 9\}\cup \\ 
\cup \{ 10, 11, 12\}\cup\{ 13, 14, 15, 16\}\dots
\end{multline*}

Let $$\sigma=\{ 1\}\cup\{ 3, 4\}\cup\{ 7, 8, 9\}\cup\{ 13, 14, 15, 16\}\dots$$
(half the sets in the partition) and let $\sigma_j$ denote the j-th subset in this list.

By general properties of ultrafilters, either $\sigma$ or its complement is in $\ww$. We assume, without loss of generality, that $\sigma\in\ww$.

Since $\ww$ is rare, there exists $B\in \ww$ such that $B\cap \sigma_j$ has at most one element for all $j$.

Define $\gamma=\sigma \cap B\in \ww$, then
 $|\gamma \cap (\gamma+n) |\le 2n \quad \forall n\in \bb N$ which can be seen by looking at how the $\sigma_j$'s in $\sigma$ are spread out.

Thus, if $U$ denotes the clopen set in $\beta \bb N$ corresponding to $\gamma$, then $(n+U) \cap (m+U) \subset \bb N$ and is finite. Thus, the relatively open set $U \cap \bb N^*$ is wandering in $\bb N^*.$

\end{proof}

We now present Davidson's proof for general groups.

\begin{thm} Let $G$ be a countable, discrete group and let $\ww \in \BG$ be a rare ultrafilter, then $\ww$ is wandering in the corona $G^*,$ and, hence $\cl A_{\ww} = \cl D \times_r G.$
\end{thm}
\begin{proof}  Let $e$ denote the identity of $G$ and choose finite subsets, $\{e \} = G_0 \subseteq G_1 \subseteq \ldots$ with $G_n = G_n ^{-1}$ and $G = \cup_{n=0}^{\infty} G_n.$ Set $P_n= G_n \cdots G_1 \cdot G_0$ and let $A_n = P_n \backslash P_{n-1}, n \ge 1,$ and $A_0 = \{ e \}.$ We have that each $A_n$ is finite, $A_n \cap A_m$ is empty for $n \ne m$, and $G = \cup_{n=0}^{\infty} A_n.$

If $g \in G_k$ and $n > k,$ we claim that $gA_n  \subseteq A_{n-1} \cup A_{n} \cup A_{n+1} = P_{n+1} \slash P_{n-2}.$ To see this note that $gP_n \subseteq P_{n+1}.$ If $p \in P_n$ and $gp \in P_{n-2}$, then $p = g^{-1}(gp) \in P_{n-1}$ and, hence $p \notin A_n.$  Thus, if $p \in A_n,$ then $gp \notin P_{n-2},$ and the claim follows.

Now since $\ww$ is rare, we may pick a set in $U$ in the ultrafilter $\ww$, such that   
$U \cap A_n$ has at most one element for all $n.$  Let $E = \cup A_n, n$ even and $O = \cup A_n, n$ odd. Then $E \cap O$ is empty and $E \cup O =G.$ Hence either $E$ is in the ultrafilter, or $O$ is in the ultrafilter.  

If $E$ is in the ultrafilter $\ww,$ let $V = U \cap E$ which is in the ultrafilter $\ww$.  If $g \in G_k, g \ne e$ we claim that the cardinality of $gV \cap V$ is at most $k/2$. Assuming the claim, we see that the clopen nieghborhood $\cl V$ of $\ww$ corresponding to $V$, has the property that, $$\cl W =\cl V \cap \bb G^*$$ is an open neighborhood of $\ww$ with $g \cl W \cap \cl W$ empty for all $g \ne e.$

To see the claim, write $V= \{ a_{2m}: a_{2m} \in A_{2m}, m \ge 0 \},$ so that $x \in gV \cap V$ if and only if $x = ga_{2m} = a_{2j}.$ But if $2m >k,$ then $ga_{2m} \in A_{2m-1} \cup A_{2m} \cup A_{2m+1},$ and so $j=m,$ forcing $g=e.$
Thus, $ga_{2m} = a_{2j}$ has no solutions for $2m > k$ when $g \ne e.$

The proof for the case that $O$ is in the ultrafilter $\ww,$ is identical.

\end{proof}

\begin{cor} Let $G$ be a countable discrete group. Assuming the continuum hypothesis, $G^*$ contains a dense
  set of wandering points.
\end{cor}
\begin{proof} By \cite{C}, if we assume the continuum hypothesis, then
  the rare ultrafilters
  are dense in $G^*.$
\end{proof}

We don't know if it is necessary to assume the continuum hypothesis to
conclude that the wandering points are dense in $G^*.$

\begin{prob} Let $\ww$ be an ultrafilter on $\bb N$, then for every countable, discrete group $G$ and every labeling, $G= \{ g_n: n \in \bb N \}$, we have that $\ww$ determines a point in $G^*.$ The above result shows that when $\ww$ is rare then the point obtained in this fashion is wandering for the action of $G$ on $G^*.$ Conversely, if $\ww$ is an ultrafilter on $\bb N$ with this property, then must $\ww$ be rare?
\end{prob}

By the result of Reid\cite{R}, we know that the state given by any rare ultrafilter has a unique extension to $B(\ell^2(G)).$ This suggests the following conjecture:

\begin{conj} If $\ww \in G^*$ is wandering for the $G$ action on $G^*,$ then the state corresponding to evaluation at $\ww$ extends uniquely to a state on $B(\ell^2(G)).$ 
\end{conj} 

\begin{prop} Let $G$ be a countable, discrete group. Then
$\delta$-stable ultrafilters are non-recurrent in $\BG$.
\end{prop}
\begin{proof} Let $\ww$ be a $\delta$-stable ultrafilter, we have that  $\{g \cdot \ww: g \in G \}$ is a distinct set of points.

For each $g \in G$ that is not equal to the identity, the complement of $\{ g \cdot \ww \}$ is an
open neighborhood of $\ww.$ The intersection of these sets is a
$G_{\delta}$ containing $\ww$ and hence, applying the equivalence of $\delta$-stable to P-point, contains an open neighborhood
of $\ww.$ No point on the orbit of $\ww$ returns to this open neighborhood. 
\end{proof}

Combining the above result with Theorem~\ref{non-recurrent}, we see
that $\delta$-stable ultrafilters satisfy the condition of
Corollary~\ref{condition2} that is necessary for $s_{\ww}$ to have a unique extension.

\end{document}